\documentclass[leqno]{amsart}
\usepackage{amsmath}
\usepackage{amssymb}
\usepackage{amsthm}
\usepackage{enumerate}
\usepackage[mathscr]{eucal}
\theoremstyle{plain}
\usepackage{tikz}
\newtheorem{theorem}{Theorem}[section]
\newtheorem{lemma}[theorem]{Lemma}
\newtheorem{prop}[theorem]{Proposition}
\theoremstyle{definition}
\newtheorem{definition}[theorem]{Definition}
\newtheorem{remark}[theorem]{Remark}

\newtheorem{example}[theorem]{Example}

\newtheorem{cor}[theorem]{Corollary}
\theoremstyle{remark}

\begin{document}
	\title[Some special subspaces from the perspective of best coapproximation]{On some special subspaces of a Banach space, from the perspective of best coapproximation}
	\author[Sohel, Ghosh, Sain and  Paul  ]{Shamim Sohel, Souvik Ghosh, Debmalya Sain and Kallol Paul }

	\newcommand{\acr}{\newline\indent}

		\address[Sohel]{Department of Mathematics\\ Jadavpur University\\ Kolkata 700032\\ West Bengal\\ INDIA}
	\email{shamimsohel11@gmail.com}
	
	\address[Ghosh]{Department of Mathematics\\ Jadavpur University\\ Kolkata 700032\\ West Bengal\\ INDIA}
	\email{sghosh0019@gmail.com}
	
	\address[Sain]{Department of Mathematics\\ Indian Institute of Information Technology, Raichur\\ Karnataka 584135 \\INDIA }
	\email{saindebmalya@gmail.com}
	
	\address[Paul]{Department of Mathematics\\ Jadavpur University\\ Kolkata 700032\\ West Bengal\\ INDIA}
	\email{kalloldada@gmail.com, kallol.paul@jadavpuruniversity.in}

		\thanks{ The first and second author would like to thank  CSIR, Govt. of India, for the financial support in the form of Senior Research Fellowship under the mentorship of Prof. Kallol Paul.  
	}


	\maketitle
	
	\begin{abstract}
		We study the best coapproximation problem in Banach spaces, by using Birkhoff-James orthogonality techniques. We introduce two special types of subspaces, christened the anti-coproximinal subspaces and the strongly anti-coproximinal subspaces. We obtain a necessary condition for the strongly anti-coproximinal subspaces in a reflexive Banach space whose dual space satisfies the Kadets-Klee Property. On the other hand, we provide a sufficient condition  for the strongly anti-coproximinal subspaces  in a general  Banach space.  We also characterize the anti-coproximinal subspaces of a smooth Banach space.  Further, we  study these special subspaces in a finite-dimensional polyhedral Banach space and find some interesting geometric structures associated with them.  
	\end{abstract}
	\section{Introduction.}

    Franchetti and Furi introduced the best coapproximation problem in the framework of Banach spaces in \cite{FF}, as a complementary notion to the well-known best approximation problem. Unlike its counterpart, the best coapproximation problem remains a less explored area of research, especially from the computational point of view.  Indeed, the difficulty of the problem arises from the fact that the existence of best coapproximation is not guaranteed even in finite-dimensional Banach spaces. Interested readers are referred to \cite{LT, N, PS, RS, W} for more information on this topic. The concept of contractive mapping is closely connected with the idea of best coapproximation, to delve deeper the readers can go through \cite{Bruck, Bruck1, KL}.  Recently, an algorithmic approach has been given in \cite{SSGP} and \cite{SSGP2} to study the best coapproximation in $\ell_\infty^n$ and $\ell_1^n,$ respectively. The purpose of this article is to identify some special types of subspaces of a Banach space, which may be regarded as the least favorable candidates for the purpose of the best coapproximation problem. We also illustrate that such subspaces enjoy several geometric properties unique to them. Having described the motivation behind this study, let us now introduce the relevant notations and terminologies.\\

Let $ \mathbb{X}, \mathbb{{Y}}$ denote real Banach spaces and let $ \mathbb{H}$ denote a real Hilbert space. We use the notations $B_{\mathbb{X}}$ and $S_{\mathbb{X}}, $  for the unit ball $ B_\mathbb{X} := \{ x \in \mathbb{X}: \|x\| \leq 1\} $ and the unit sphere $ S_\mathbb{X} := \{  x \in \mathbb{X}: \|x\| = 1\} $ of $ \mathbb{X}, $ respectively.  The dual space of $\mathbb{X}$ is denoted by $\mathbb{X}^*.$
	The annihilator of a subspace $\mathbb{Y}$ of $\mathbb{X}$ is defined as $	\mathbb{{Y}}^{\perp} := \{ f\in \mathbb{{X}}^*: f(y)= 0, ~\text{for each } y \in \mathbb{{Y}}\}.$ For  a subspace $\mathbb{Z}$ of $\mathbb{X}^*,$ the pre-annihilator of $\mathbb{Z}$ is defined as $^{\perp}\mathbb{Z} := \{ x\in \mathbb{{X}}: f(x)= 0, ~\text{for each } f\in \mathbb{Z}\}.$
	For a  non-empty convex subset $ A $ of $\mathbb{X},$ an element $ z \in A$ is said to be an extreme point of $A$ if whenever $z = (1-t)x + ty,$ for some $t \in (0, 1)$ and some $ x, y \in A,$ then $ x = y = z.$ The collection of all the extreme points of $ A$ is denoted as $ Ext(A).$
	A Banach space $ \mathbb{X}$ is said to be strictly convex if $Ext(B_{\mathbb{X}}) = S_\mathbb{X}.$ It is easy to observe that strict convexity of $ \mathbb{X} $ is equivalent to the geometric condition that $ S_{\mathbb{X}} $ does not contain any non-trivial straight line segment. We recall that a point $x\in S_\mathbb{X}$ is said to be a rotund point \cite{G} of $B_\mathbb{X}$ if $\|y\|= \|\frac{x+y}{2}\| = 1$ implies that $x=y.$ Clearly, if every point of $S_\mathbb{X}$ is rotund then $\mathbb{X}$ is strictly convex.
	Given any non-zero $ x \in \mathbb{X},$ $ f \in S_{\mathbb{X}^*} $ is said to be a support functional of $x$  if $ f(x)= \|x\|.$ The set of all support functional(s) of a nonzero $ x \in \mathbb{X}$ is written as $ J(x) := \{ f \in S_{\mathbb{X}^*}: f(x)= \|x\|\}.$ For any nonzero $ x \in \mathbb{X},$ $ x$ is said to be smooth if $J(x) $ is singleton. A Banach space $ \mathbb{X}$ is said to be smooth  if each  $x\in S_\mathbb{X}$ is smooth. The collection of all smooth points in $\mathbb{{X}}$ is denoted by $ Sm(\mathbb{X}).$  For a subspace $ \mathbb{Y} $ of $ \mathbb{X}, $ let 
	$\mathcal{J}_{\mathbb{Y}} = \{ f \in S_{\mathbb{X}^*} : f(y) = 1, \text{for some}~  y \in Sm(\mathbb{{X}})\cap S_{\mathbb{Y}} \}.$ It is easy to check that $\mathcal{J}_{\mathbb{Y}} \subseteq Ext(B_{\mathbb{X}^*}).$ Whenever $ Sm(\mathbb{X}) \cap S_\mathbb{{Y}}= \emptyset, $ we define $\mathcal{J}_\mathbb{{Y}}= \emptyset.$ 
 Let us recall the well known definition of  the modulus of smoothness of a Banach space $\mathbb{X}(\neq \{0\})$, which is denoted  $\rho_{\mathbb{X}}(t), $ and is defined as follows:
		\begin{eqnarray*}
	\rho_{\mathbb{X}}(t) &=& \sup \bigg\{\frac{\|x+ t y\|+\|x- t y\|}{2} -1 : x, y \in S_\mathbb{X}\bigg\},
	\end{eqnarray*}
where $ t \in (0, +\infty).$ The space 	$\mathbb{X} (\neq \{0\})$ is said to be uniformly smooth (\cite[Def. 5.5.2]{M}) if	 $\frac{\rho_{\mathbb{X}}(t)}{t}\to 0,$	  as $t \to 0^+.$

	 Given any $f \in \mathbb{X}^*,$ the kernel of $f$ is denoted by $ker~f:= \{ x \in \mathbb{{X}}: f(x)=0\}.$ A finite-dimensional Banach space
	$\mathbb{X}$ is said to be a polyhedral Banach space if $Ext(B_{\mathbb{{X}}}) $ is finite. A convex set $F\subset S_\mathbb{X}$ is said to be a face of $B_\mathbb{X}$ if for some $x_1, x_2 \in S_\mathbb{X},$ $(1-t)x_1+tx_2 \in F$ implies that $x_1, x_2 \in F, $ where $0<t<1.$ A maximal face  of $ B_{\mathbb{X}} $ is said to be a facet. We use the notation $int(F)$ to denote the interior of a face $F$ endowed with the usual subspace topology of $F.$ Given a subset $M$ of $\mathbb{X}^*,$ $\overline{M}^{w^*}$ denotes the closure of $M$ with respect to the weak*- topology defined on $\mathbb{X}^*.$   We also recall that $\mathbb{X}$ satisfies the Kadets-Klee Property if whenever $\{x_n\}$ is a sequence in $\mathbb{X}$ and $x \in \mathbb{X}$ such that $x_n \overset{w}{\to} x$ and $\|x_n \| \to \|x\|,$ it follows that $x_n \to x.$
	Next we recall the definition of  best coapproximation in Banach spaces.

	\begin{definition} \cite{FF}
		Let $ \mathbb{Y} $ be a subspace of Banach space $ \mathbb{X}. $ Given any $ x \in \mathbb{X}, $ we say that $ y_0 \in \mathbb{Y} $ is a best coapproximation to $ x $ out of $ \mathbb{{Y}} $ if $ \| y_0 - y \| \leq \| x - y \|, $ for all $ y \in \mathbb{Y}. $
	\end{definition}
	
	Given any $x\in \mathbb{{X}}$ and a subspace $\mathbb{{Y}}$ of $\mathbb{X},$ $\mathcal{R}_\mathbb{{Y}}(x)$ denotes the (possibly empty)  set of all best coapproximations to $x$ out of $\mathbb{{Y}}.$  A subspace $\mathbb{{Y}}$ of the Banach space $\mathbb{{X}}$ is said to be coproximinal if for any $x\in \mathbb{X},$  $\mathcal{R}_\mathbb{{Y}}(x)\neq \emptyset.$
	The concept of  Birkhoff -James orthogonality plays a vital role  in the  study  of the best coapproximation problem. Given $ x, y \in \mathbb{X}, $ we say that $ x $ is Birkhoff-James orthogonal \cite{B,J} to $ y, $ written as $ x \perp_B y, $ if $ \| x+\lambda y \| \geq \| x \|, $ for all $ \lambda \in \mathbb{R}. $ As mentioned in \cite{FF}, it is easy to observe that given any subspace $ \mathbb{Y} $ of a Banach space $ \mathbb{X} $ and an element $ x \in \mathbb{X}, $ $ y_0 \in \mathbb{Y} $ is a best coapproximation to $ x $ out of $ \mathbb{Y} $ if and only if $ \mathbb{Y} \perp_B (x-y_0), $ i.e., $ y \perp_B (x-y_0) ,$ for all $ y \in \mathbb{Y}.$  We also need the notion of approximate Birkhoff-James orthogonality in our study. 
		In \cite{D}, Dragomir first introduced the approximate Birkhoff-James orthogonality in the following way:\\
	Let $\epsilon \in [0,1).$ Then for $x, y\in \mathbb{X},$ $x$ is said to be approximate $\epsilon$-Birkhoff-James orthogonal to $y$  if for each $\lambda \in \mathbb{R},$ the following holds:
	 \[\|x+\lambda y\| \geq (1-\epsilon)\|x\|.\] 
	Later on Chmieli\'nski \cite{C05} introduced another version of the approximate Birkhoff-James orthogonality. Let $ \epsilon \in [0,1).$  Given any $ x, y \in \mathbb{X},$ we say that $x$ is said to be approximate orthogonal to $y,$ written as $ x \perp_B^\epsilon y,$ if 
	\[ \|x + \lambda y\|^2 \geq \|x\|^2 - 2 \epsilon\|x\|\|\lambda y\|.
	\]
	Recently \cite{C}, an equivalent definition of the approximate orthogonality has been obtained:
	\[x \perp_B^\epsilon y \iff \|x + \lambda y\| \geq \|x\| -  \epsilon \|\lambda y\|,~ \text{for ~every ~} ~\lambda \in \mathbb{R}.
	\]
	In view of this characterization of approximate orthogonality, it is natural to consider  the following approximate version of the best coapproximation problem:  
	
	\begin{definition}
		Let $ \mathbb{Y} $ be a subspace of Banach space $ \mathbb{X}. $ Let $ \epsilon \in [0,1).$  Given any $ x \in \mathbb{X}, $ we say that $ y_0 \in \mathbb{Y} $ is an  $\epsilon$-best coapproximation to $ x $ out of $ \mathbb{{Y}} $ if $ \mathbb{{Y}} \perp_B^\epsilon x- y_0, $ i.e., $ y \perp_B^\epsilon (x-y_0) ,$ for all $ y \in \mathbb{Y}. $ 
	\end{definition}

	The primary purpose of our study is to investigate the least favorable scenario that can arise in studying the best coapproximation problem. Accordingly, we introduce the following two types of subspaces of a Banach space from the perspective of best coapproximation and $\epsilon$-best coapproximation.
	
	\begin{definition}\label{epsilon}
		Let $ \mathbb{Y} $ be a subspace of Banach space $ \mathbb{X}. $ Then 
		\begin{itemize}
			\item[(i)] $\mathbb{{Y}}$ is said to be an \textit{anti-coproximinal subspace} of $\mathbb{X}$ if for any given $x \in \mathbb{X} \setminus \mathbb{Y},$ there does not exist any best coapproximation to $x$ out of $\mathbb{Y}.$
			
			\item[(ii)] $\mathbb{Y}$ is said to be a \textit{strongly anti-coproximinal subspace} of $\mathbb{X}$ if for any given $x \in \mathbb{X} \setminus \mathbb{Y}$ and for any $\epsilon \in [0, 1),$ there does not exist $\epsilon$ -best coapproximation to $x$ out of $\mathbb{Y}.$
		\end{itemize}
	\end{definition}
  
	It is easy to observe that if  $\mathbb{Y}$ is  (strongly) anti-coproximinal in $\mathbb{X}$ and $\mathbb{Z}$ is a subspace of $\mathbb{X}$ containing $\mathbb{Y}$ then $\mathbb{Y}$ is (strongly) anti-coproximinal in $\mathbb{Z}.$ However, it is quite obvious that if $\mathbb{Y}$ is (strongly) anti-coproximinal in $\mathbb{Z},$ it does not imply that $\mathbb{Y}$ is (strongly) anti-coproximinal in $\mathbb{X}$ (see Remark \ref{rem}). Therefore, it is important to specify the mother space whenever we consider the anti-coproximinal and the strongly anti-coproximinal subspaces.
	
		We note from \cite[Example 2.13]{SSGP} that if a subspace $\mathbb{Y}$ is not  coproximinal in $\mathbb{X}$ then it is not necessarily true that $\mathbb{Y}$ is an anti-coproximinal subspace of $\mathbb{X}$. On the other hand, if $\mathbb{Y}$ is a  strongly anti-coproximinal subspace of $\mathbb{X},$ it  implies that  $\mathbb{Y}$ is an  anti-coproximinal subspace of $\mathbb{X}$.  However, as we will observe in Example \ref{example2}, the converse of this fact is not true. It should be noted that every one-dimensional subspace is coproximinal in any Banach space \cite[Lemma 1]{FF}. Whenever the anti-coproximinal and the strongly anti-coproximinal subspaces are concerned, we only consider  the proper subspaces of dimension strictly greater than one.  In order to facilitate the understanding of these two special types of subspaces of a Banach space, let us consider the following simple example:
	
	
	\begin{example}
		Let $\widetilde{v}_1= (3,0,2), \widetilde{v}_2= (0,3,2) \in \ell_{\infty}^3$ and let $\mathbb{Y}= span \{ \widetilde{v}_1, \widetilde{v}_2\}.$ Applying \cite[Th. 2.10]{SSGP} along with some straightforward calculations, we observe that for any $x \in \ell_{\infty}^3 \setminus \mathbb{Y}$, there does not exist a best coapproximation to $x$ out of $\mathbb{Y}.$ In other words, $\mathbb{Y}$ is an anti-coproximinal subspace of $ \ell_{\infty}^3. $ We will see that the same subspace happens to be a strongly anti-coproximinal subspace of $\ell_{\infty}^3$ as well. On the other hand, we will also present several explicit examples of the anti-coproximinal subspaces of a Banach space, which are not strongly anti-coproximinal in that Banach space.
		
	\end{example}

	This article is divided into two sections including the introductory one. The main results are separated in two subsections, viz., Section I and Section II. In Section I, we characterize the anti-coproximinal subspaces in a smooth Banach space and using that we characterize the Hilbert spaces among smooth Banach spaces.   We also present a necessary condition  for a subspace to be strongly anti-coproximinal in a reflexive Banach space, whose dual space satisfies the Kadets-Klee Property. On the other hand, a sufficient condition for the same is obtained in any Banach space. In  Section II, we exclusively consider the finite-dimensional polyhedral Banach spaces. We characterize the strongly anti-coproximinal subspaces in a finite-dimensional polyhedral Banach space and show that the strongly anti-coproximinal subspaces of such spaces possess some nice geometric structures. We study the spaces $\ell_{\infty}^n$ and $\ell_{1}^n$ separately and characterize the anti-coproximinal and the strongly anti-coproximinal subspaces of these distinguished polyhedral Banach spaces, which is especially advantageous from the computational point of view.  \\

	\section{Main Results.}
	
	\section*{section-I}
		
We begin this section with the following two well-known observations.
	
	\begin{theorem}\cite[Th. 2.1]{PS} \label{PS}
		Let $ \mathbb{Y} $ be a subspace of a Banach space $ \mathbb{X} $ and let $x\in \mathbb{X}.$ Then $y_0\in \mathbb{Y}$ is a best coapproximation to $x$ out of $\mathbb{Y}$ if and only if  for any $y \in \mathbb{{Y}},$ there exists $f_y\in J(y)$ such that $f_y(x-y_0)=0.$ 
	\end{theorem}
	
	\begin{theorem}\cite[Th. 2.3]{C}\label{chm}
		Let $\mathbb{X}$ be a  Banach space. Let $x, y\in \mathbb{X}$ and let $\epsilon \in [0, 1).$ Then the following conditions are equivalent:
		\begin{itemize}
			\item[(i)] $x \perp_B^{\epsilon} y$
			
			\item[(ii)] there exists $f\in J(x)$ such that $|f(y)|\leq \epsilon\|y\|$
		\end{itemize}
	\end{theorem}
	
	Using Theorem \ref{PS}, a simple characterization of the  anti-coproximinal subspace follows immediately.
	
	\begin{prop}
		Let $ \mathbb{Y} $ be a subspace of Banach space $ \mathbb{X}. $ Then	$\mathbb{Y}$ is an anti-coproximinal subspace of $\mathbb{X}$ if and only if for any $x \in \mathbb{X}\setminus \mathbb{Y},$ there exists a $y_0 \in \mathbb{Y}$ such that $x \notin ker~f,$ for any $f \in J(y_0).$
	\end{prop}

		We next obtain a characterization of the $\epsilon$-best coapproximation, by applying Theorem \ref{chm}.
	
	\begin{prop}\label{approximate}
		Let $ \mathbb{Y} $ be a subspace of a Banach space $ \mathbb{X}. $ Then the following statements are equivalent:
		\begin{itemize}
			\item[(i)] $y_0$ is an $\epsilon$-best coapproximation to $x$ out of $\mathbb{{Y}}$
			\item [(ii)] for any $y \in \mathbb{{Y}},$ $\|x-y\| \geq \| y_0-y\| - \epsilon \|x-y_0\|$
			\item[(iii)] for any $y \in \mathbb{{Y}},$ there exists $ f_y \in J(y)$ such that $ |f_y(x-y_0)| \leq \epsilon\|x-y_0\|.$
		\end{itemize}
		
	\end{prop}
	
	\begin{proof}
		
		$ (i) \Rightarrow (ii)$ :
		Let $y_0$ be an $\epsilon$-best coapproximation to $x$ out of $\mathbb{{Y}}.$ Therefore, $ \mathbb{{Y}} \perp_B^\epsilon (x-y_0),$ i.e., $y \perp_B^\epsilon (x-y_0),$ for any $y \in \mathbb{Y}.$ Thus for any $y\in \mathbb{{Y}},$ 
		$ \|y +  \lambda(x-y_0) \| \geq \|y\|- \epsilon \|\lambda(x-y_0)\|,
		$
		for all $\lambda \in \mathbb{R}.$ Putting $\lambda =1,$ we get
	$	\|y +  (x-y_0) \| \geq \|y\|- \epsilon \|(x-y_0)\| ,$  from which the desired inequality follows easily. \\
			$ (ii) \Rightarrow (iii)$ : 
		For any $y \in \mathbb{{Y}}$ and for any nonzero $\lambda \in \mathbb{R},$
		\begin{eqnarray*}
			\| y + \lambda (x-y_0)\| &=& |\lambda| \|x - (y_0-\frac{1}{\lambda} y)\| \\ &\geq & |\lambda|  \{\| y_0-(y_0-\frac{1}{\lambda} y) \| - \epsilon \| x- y_0\|\}\\
			&=& |\lambda |( \|\frac{1}{\lambda} y\| -  \epsilon \| x-y_0\|) \\ 
			&=& \|y\| - \epsilon \|\lambda (x-y_0)\|.
		\end{eqnarray*}
		For $\lambda=0,$ the above inequality holds trivially. This implies $y \perp_B^\epsilon (x-y_0),$ for any $y \in \mathbb{Y}$  and so the result follows from Theorem \ref{chm}. \\
		$ (iii) \Rightarrow (i)$ : 
	Following Theorem \ref{chm}, we get $y \perp_B^{\epsilon} x-y_0,$ for all $y \in \mathbb{Y}.$ Thus we obtain $\mathbb{Y}\perp_B^{\epsilon} x-y_0.$
	\end{proof}
	

	Throughout the next part of this section, we focus on the anti-coproximinality and the strong anti-coproximinality of the closed subspaces in a Banach space. In the following proposition, we observe that every dense subspace of a Banach space is strongly anti-coproximinal in that space  and consequently, anti-coproximinal too.

	\begin{prop}
		Let $\mathbb{Y}$ be a dense subspace of a Banach space $\mathbb{X}.$ Then $\mathbb{Y}$ is a strongly anti-coproximinal subspace of $\mathbb{X}.$
	\end{prop}
	
	\begin{proof}

		Suppose on the contrary that $\mathbb{Y}$ is not a strongly anti-coproximinal subspace of $\mathbb{X}.$ Then there exists an $\epsilon \in [0, 1)$ and an $x \in \mathbb{X}\setminus \mathbb{Y}$ such that $\mathbb{Y} \perp_B^{\epsilon} x.$ Since $\mathbb{Y}$ is dense in $\mathbb{X},$ it follows that there exists a sequence $\{y_n\}_{n \in \mathbb{N}} \subset \mathbb{Y}$ such that $y_n \to x.$ Also, we note that $y_n \perp_B^{\epsilon} x.$ Then for each $\lambda \in \mathbb{R},$ we have $\|y_n + \lambda x\| \geq \|y_n\| - \epsilon \|\lambda x\|.$  Since the norm function is continuous, letting $ n \to \infty, $ we get that $\|x+\lambda x\| \geq \|x\| - \epsilon\|\lambda x\|,$ for each $\lambda \in \mathbb{R}.$ This implies that $x \perp_B^{\epsilon} x,$ and therefore $x=0,$ which contradicts the fact that $x \in \mathbb{X} \setminus \mathbb{Y}.$

	\end{proof}

 In the next theorem we characterize the anti-comproximinal subspaces in a smooth Banach space.

	\begin{theorem}\label{anticoproximinal}
		Let $ \mathbb{Y} $ be a subspace of a smooth Banach space $ \mathbb{X}. $ Then $\mathbb{Y}$ is an anti-coproximinal subspace of $\mathbb{X}$ if and only if $\overline{span~ \mathcal{J}_{\mathbb{Y}}}^{w^*}=\mathbb{X}^*$.
	\end{theorem}

	\begin{proof}

		Let us first prove the sufficient part of the theorem. Suppose on the contrary that $\mathbb{Y}$ is not an anti-coproximinal subspace of $\mathbb{X}$. This implies that for some $x \in \mathbb{X}\setminus \mathbb{Y},$ there exists a $y_0 \in \mathbb{Y}$ such that $y_0$ is a best coapproximation to $x$ out of $\mathbb{Y}.$ Since $\mathbb{X}$ is smooth, it follows from Theorem \ref{PS} that for any $y \in \mathbb{Y},$ $f_y(x-y_0) = 0,$ where $ J(y)= \{f_y\}.$  It is easy to see that  $ \mathcal{J}_\mathbb{Y} = \{f_y : y \in \mathbb{Y}\}.$
		This implies, $g(x-y_0)=0,$ for any $g \in span ~\mathcal{J}_{\mathbb{{Y}}}.$  Now let us consider $f \in \mathbb{X}^*.$  Since $\overline{span~\mathcal{J}_\mathbb{Y}}^{w^*}=\mathbb{X}^*$, it follows that  there exists a net $\{f_\alpha\}_{\alpha \in \Lambda} \subset span~\mathcal{J}_\mathbb{Y}$ such that $f_\alpha$ is weak*-convergent to $f.$ 
		 Since $f_\alpha(x-y_0) = 0,$ for each $\alpha \in \Lambda,$ it follows that $f(x-y_0) = 0.$	 Note that $f \in \mathbb{X}^*$ is taken arbitrarily. Thus we obtain that $f(x-y_0) = 0,$ for all $f \in \mathbb{X}^*,$ i.e., $x - y_0 = 0,$ which contradicts the fact that $x \in \mathbb{X}\setminus \mathbb{Y}.$ \\	
		
		Now we prove the necessary part of the theorem.
		Suppose on the contrary that $\overline{span~\mathcal{J}_\mathbb{Y}}^{w^*} \subsetneq \mathbb{X}^*.$
		 Clearly, $^{\perp}(span~\mathcal{J}_\mathbb{Y}) = \cap_{f \in \mathcal{J}_\mathbb{Y}} ker~f.$ We note from \cite[Th. 4.7]{R} that $(^{\perp}(span~\mathcal{J}_\mathbb{Y}))^{\perp} = \overline{span~\mathcal{J}_\mathbb{Y}}^{w^*}.$ This implies that $(\cap_{f \in \mathcal{J}_\mathbb{Y}} ker~f)^{\perp} = \overline{span~\mathcal{J}_\mathbb{Y}}^{w^*}.$ Following \cite[Th. 4.7]{R}, we obtain that $(\cap_{f \in \mathcal{J}_\mathbb{Y}} ker~f)^* $ is isometrically isomorphic to $ \mathbb{X}^*/ \overline{span~\mathcal{J}_\mathbb{Y}}^{w^*}.$ Therefore, $(\cap_{f \in \mathcal{J}_\mathbb{Y}} ker~f)^* \neq 0,$ which implies that $\cap_{f \in \mathcal{J}_\mathbb{Y}} ker~f \neq 0.$ 
		Moreover, it is straightforward to see that $ (\cap_{f \in \mathcal{J}_\mathbb{Y}} ker~f ) \cap \mathbb{Y} = \{0\}.$ So, there exists a $z \in \mathbb{X}\setminus \mathbb{Y} $ such that $z \in \cap_{f \in \mathcal{J}_\mathbb{Y}} ker~f. $ Since $\mathbb{X}$ is smooth, applying \cite[Th. 2.1]{J}, we obtain that $\mathbb{Y} \perp_B z.$ Thus $0$ is a best coapproximation to $z$ out of $\mathbb{Y}.$ This contradicts the fact that $\mathbb{Y}$ is an anti-coproximinal subspace of $\mathbb{X}$. This proves the necessary part and completes the proof of the theorem.
	
	\end{proof}

	If $\mathbb{X}$ is finite-dimensional then we have the following corollary.
	
	\begin{cor}\label{cor}
	Let $\mathbb{Y}$ be a subspace of an $n$-dimensional smooth Banach space $\mathbb{X}.$ Then $\mathbb{Y}$ is an anti-coproximinal subspace of $\mathbb{X}$ if and only if $dim(span~ \mathcal{J}_\mathbb{Y})=n.$
\end{cor}

Next we give an example of an anti-coproximinal subspace in $\ell_p^n,$ where $2 < p < \infty.$ We require the following well-known result which explicitly describes the support functional of an element of $\ell_p^n.$ Suppose that $\phi$ is the isometric isomorphism between  $(\ell_p^n)^*$ and $\ell_q^n,$ where $\frac{1}{p} + \frac{1}{q}=1.$

\begin{lemma}\label{support}
	
	Let $\widetilde{x}= (x_1, x_2, \ldots, x_n) \in \ell_p^n,$ where $1 <  p < \infty.$ Then 
	$J(\widetilde{x}) = \{\widetilde{f}\}, $ where $\phi (\widetilde{f})= \Big(\frac{x_1|x_1|^{p-2}}{\|\widetilde{x}\|_p^{p/q}}, \frac{x_2|x_2|^{p-2}}{\|\widetilde{x}\|_p^{p/q}}, \ldots, \frac{x_n|x_n|^{p-2}}{\|\widetilde{x}\|_p^{p/q}}\Big)\in \ell_q^n,$ $\frac{1}{p}+\frac{1}{q} = 1$ and $\|\widetilde{x}\|_p = \big(\sum_{i=1}^n |x_i|^p\big)^{\frac{1}{p}}.$ 
\end{lemma}

\begin{example}\label{example}
	Let us consider the space $\ell_p^n,$ where $ p \in (1, \infty) \setminus \{ 2 \} $ and $n \geq 3$ with $\{e_1, e_2, \ldots, e_n\}$ as the standard ordered basis. Suppose that  $\mathbb{Y}$ is a hyperspace of $\ell_p^n,$ where $\mathbb{Y} = span\{\widetilde{x}_1, \widetilde{x}_2, \ldots, \widetilde{x}_{n-1}\},$  $\widetilde{x}_1= (1, 1, 1, 0, \ldots, 0), \widetilde{x}_2 = (1, 2, 3, 0, \ldots, 0),  \widetilde{x}_3 = e_4, \ldots, \widetilde{x}_{n-1} = e_n. $ Let $J(\widetilde{x}_i)= \{\widetilde{f}_i\},$ for any $ 1 \leq i \leq n.$ Clearly, $\widetilde{f}_i \in \mathcal{J}_{\mathbb{Y}}.$ Applying Lemma \ref{support}, we obtain the following: \\
	$ \phi(\widetilde{f}_1) = \bigg(\frac{1}{3^{1-\frac{1}{p}}}, \frac{1}{3^{1-\frac{1}{p}}}, \frac{1}{3^{1-\frac{1}{p}}}, 0, \ldots, 0 \bigg) \in \ell_q^n,$\\
	$ \phi(\widetilde{f}_2) = \bigg(\frac{1}{(1+2^p+3^p)^{1-\frac{1}{p}}}, \frac{2^{p-1}}{(1+2^p+3^p)^{1-\frac{1}{p}}}, \frac{3^{p-1}}{(1+2^p+3^p)^{1-\frac{1}{p}}}, 0, \ldots, 0 \bigg) \in \ell_q^n,$\\
	and $ \phi( \widetilde{f}_k) = e_{k+1} \in \ell_q^n,$ for all $3 \leq k \leq n-1.$ Consider the element $ \widetilde{x}_n = 3\widetilde{x}_1 - \widetilde{x}_2 \in \mathbb{Y}$ and let $ J( \widetilde{x}_n)= \{ \widetilde{f}_n\}.$ Again using Lemma \ref{support}, we get $ \phi (\widetilde{f}_n) = \Big(\frac{2^{p-1}}{(2^p+1)^{1-\frac{1}{p}}}, \frac{1}{(2^p+1)^{1-\frac{1}{p}}}, 0, \ldots, 0\Big) \in \ell_q^n.$ A straightforward computation shows that $ \{ \widetilde{f}_1, \widetilde{f}_2, \ldots, \widetilde{f}_n\}$ is a linearly independent set. Therefore, $dim(span~\mathcal{J}_\mathbb{Y}) = n.$ Thus from Corollary \ref{cor},  we conclude that $\mathbb{Y}$ is an anti-coproximinal subspace of $\ell_p^n,$ where $ p \in (1, \infty) \setminus \{ 2 \}. $ By a similar computation, it can be shown that $\mathbb{W}= \underbrace{\mathbb{{Y}} \oplus \mathbb{{Y}}\oplus \ldots \oplus \mathbb{Y}}_{r-times}$ is an anti-coproximinal subspace of $\ell_p^{rn}.$ \\
	
\end{example}

It is easy to note from Example \ref{example} that for every $ p \in (1, \infty) \setminus \{ 2 \} $ and for every $ n>2,$ there exists an anti-coproximinal subspace of $\ell_p^n.$ Combining  with the well known fact that \emph{every closed subspace of a Hilbert space is coproximinal} (\cite[Lemma. 1]{FF}), we can characterize the Hilbert space $\ell_2^n$ among the $\ell_p^n$ spaces.

\begin{theorem}
	Let $\mathbb{X} = \ell_p^n,$ where $ 1 < p < \infty$ and $ n \geq 3.$ Then $p=2$ if and only if there does not exist any anti-coproximinal subspace in $\mathbb{X}.$ 
\end{theorem}

The previous result can be further extended to characterize the Hilbert space among smooth Banach spaces having dimension at least $ 3. $

\begin{theorem}
	Let $\mathbb{X}$ be a smooth Banach space and let $dim(\mathbb{X}) \geq 3.$ Then $\mathbb{X}$ is Hilbert space if and only if there does not exist any anti-coproximinal  closed hyperspace in $\mathbb{X}$.
\end{theorem} 

\begin{proof}
	Since the necessary part  follows from \cite[Th. 1]{FF}, we only prove the sufficient part. Suppose that $\mathbb{Y}$ is a closed hyperspace of $\mathbb{X}.$ Since $\mathbb{Y}$ is not an anti-coproximinal subspace of $\mathbb{X}$, it follows from Theorem \ref{anticoproximinal} that $\overline{span~\mathcal{J}_\mathbb{Y}}^{w^*} \subsetneq \mathbb{X}^*.$ Now following  similar arguments  as given in the proof of the necessary part of Theorem \ref{anticoproximinal}, we observe that there exists an element $z \in \mathbb{X} \setminus \mathbb{Y} $ such that $\mathbb{Y} \perp_B z.$ Following \cite[Th. 4]{J2}, we conclude that $\mathbb{X}$ is a Hilbert space. 
\end{proof}

\begin{remark}
	We recall that every one-dimensional subspace of a Banach space is coproximinal. Therefore, it is clear that the previous result is not valid for $ n = 2. $
\end{remark}

 Our next goal is to separately present a necessary condition and a sufficient condition for strongly anti-coproximinal subspaces of a Banach space. First we give the sufficient condition.

	\begin{theorem}\label{sufficient; strong}
		Let $ \mathbb{Y} $ be a subspace of a Banach space $ \mathbb{X}. $ Then $\mathbb{Y}$ is  a strongly anti-coproximinal subspace of $\mathbb{X}$ if for each  $x \in \mathbb{X} \setminus \mathbb{Y},$  there exists a $y \in \mathbb{Y}$ such that $J(y) \subseteq J(x) \cup J(-x).$
	\end{theorem}
	
	\begin{proof}
		Suppose on the contrary that $\mathbb{Y}$ is not a strongly anti-coproximinal subspace of $\mathbb{X}.$ Therefore, there exists an $x \in \mathbb{X}\setminus \mathbb{Y}$ such that $y_1 \in \mathbb{Y}$ is an $\epsilon$-best coapproximation to $x$ out of $\mathbb{Y},$ for some $\epsilon \in [0,1).$ Applying Proposition \ref{approximate}, we obtain that for each $y \in \mathbb{Y},$ there exists an $f_y \in J(y)$ such that $ |f_y(x - y_1)| \leq \epsilon \|x - y_1\|< \|x-y_1\|.$  Therefore,  for each $y \in \mathbb{Y},$ there exists an $f_y \in J(y),$ such that $f_y \notin J(x - y_1) \cup J(-(x-y_1)).$ This contradicts the hypothesis of  the theorem, thereby finishing the proof. 
	\end{proof}

Let us now present a necessary condition for strongly anti-coproximinal subspaces of a Banach space under some additional nice conditions.

\begin{theorem}\label{necessary; strong}
	Let $ \mathbb{X} $ be a reflexive  Banach space and let $\mathbb{X}^*$ satisfies the Kadets-Klee Property. Let $\mathbb{Y}$ be a closed subspace of $\mathbb{X}.$ If $\mathbb{Y}$ is a  strongly anti-coproximinal subspace of $\mathbb{X}$ then for each  $x \in \mathbb{X} ,$  there exists an element  $y \in \mathbb{Y}$ such that $J(y) \cap J(x) \neq \emptyset.$
\end{theorem}

\begin{proof}
	Note that whenever $x \in \mathbb{Y},$ we have nothing to prove. Take $x \in \mathbb{X} \setminus \mathbb{Y}.$ 	Since $\mathbb{Y}$ is a strongly anti-coproximinal subspace of $\mathbb{X}, $ it follows that for any $\epsilon \in [0,1),$ $ \mathbb{Y} \not\perp_B^\epsilon x.$ Let us take  a sequence $\{\epsilon_n\}_{n \in \mathbb{N}} \subset [0,1) \to 1,$ as $n \to \infty.$ Suppose that  for each $n \in \mathbb{N},$  there exists $y_n \in S_{\mathbb{Y}}$ such that $y_n \not \perp_B^{\epsilon_n} x.$
		 From Theorem \ref{chm} we obtain that   for any $f_n \in J(y_n),$ $|f_n(x)| > \epsilon_n \|x\|.$  Since $\mathbb{X}$ is reflexive, it follows that $\mathbb{X}^*$ is reflexive and therefore without loss of generality we may and do assume that $f_n$ is weakly convergent to $f,$ for some $f \in B_{\mathbb{{X}}^*}.$ So, $f_n(x) \to f(x).$  Taking limit on the  both sides of the above inequality, we obtain $|f(x)| \geq \|x\|.$ Since $f \in B_{\mathbb{X}^*},$ $|f(x)|=\|x\|,$ and therefore,  $\|f\|=1.$ Thus  either $f \in J(x)$ or $-f \in J(x).$ Also $ \|f_n \| \to \|f\|$ as $ n \to \infty.$  Since $\mathbb{{X}}^*$ satisfies the Kadets-Klee Property, it follows that $f_n \to f$ as $n \to \infty.$ As $\mathbb{X}$ is reflexive and $\mathbb{Y}$ is a closed subspace of $\mathbb{X},$ then $\mathbb{Y}$ is also reflexive, and therefore $B_{\mathbb{Y}}$ is weakly compact. So, $y_n$ weakly converges to $y,$ for some $y \in B_{\mathbb{Y}}.$ Now as $f_n \to f$ and $y_n \overset{w}{\to} y,$ it is straightforward to see that $f(y) =1.$ Therefore, $f \in J(y)$ and consequently,  either $J(y) \cap J(x) \neq \emptyset$ or $J(-y) \cap J(x) \neq \emptyset.$  This completes the theorem.

\end{proof}

\begin{remark}
Observe that the above condition is only necessary but not sufficient, see  Example \ref{not strong}. 
\end{remark}



Applying Theorem \ref{necessary; strong}, it is possible to give examples of Banach spaces which do not contain any strongly anti-coproximinal closed subspaces. 
	
	\begin{theorem}
		Let $ \mathbb{X} $ be a reflexive Banach space and let $\mathbb{X}^*$ satisfies the Kadets-Klee Property. Suppose that $\mathbb{Y}$ is a closed subspace of $\mathbb{X}$ such that there exists a rotund point in $S_\mathbb{X} \setminus S_\mathbb{Y}.$ Then $\mathbb{Y}$ is not a  strongly anti-coproximinal subspace of $\mathbb{X}.$ In particular, every reflexive strictly convex Banach space, whose dual satisfies the Kadets-Klee Property, does not contain any strongly anti-coproximinal closed subspaces.
		
	\end{theorem}
	
	\begin{proof}
		
		Suppose that  $x \in S_\mathbb{X} \setminus S_\mathbb{Y}$ is a rotund point. It is straightforward to see that for any $y \in \mathbb{Y},$ $J(x) \cap J(y) = \emptyset.$ Indeed, if possible let $y_0\in S_\mathbb{Y}$ be such that $f\in J(x) \cap J(y_0).$ This implies that $ f(\frac{x+y_0}{2})=1 \implies \|\frac{x+y_0}{2}\| = 1.$ Since $x$ is rotund, it follows that $x=y_0,$ a contradiction. Now  applying Theorem \ref{necessary; strong}, it can be concluded that $\mathbb{Y}$ is not a strongly anti-coproximinal subspace of $\mathbb{X}.$ This completes the proof of the first part. The second part follows trivially from the first part.
	\end{proof}  
	
	
	Our next result shows that the condition of strict convexity in the previous theorem can be replaced by the condition of smoothness. 
	
	\begin{theorem}\label{smooth}
		Let $ \mathbb{X} $ be a reflexive smooth Banach space and let $\mathbb{X}^*$ satisfies the Kadets-Klee Property. Suppose that  $\mathbb{Y}$ is a closed subspace of $\mathbb{X}.$ Then $\mathbb{Y}$ is not a  strongly anti-coproximinal subspace of $\mathbb{X}$.

	\end{theorem}
	
	\begin{proof}
		Suppose that $\mathbb{Y}$ is a closed subspace of $\mathbb{X}$. 
		Let us consider $  g \in \mathbb{Y}^\perp$ such that $\|g\|=1.$ Since $\mathbb{X}$ is reflexive, there exists a $z \in S_\mathbb{X}$ such that $|g(z)|=1.$ Therefore, $J(z)=\{g\}$ or $J(z)= \{-g\},$ as $\mathbb{X}$ is smooth.  We claim that for any $y \in \mathbb{Y},$  $ J(y) \cap J(z) = \emptyset.$ Otherwise, take a nonzero element $y_1 \in \mathbb{{Y}}$ such that $J(y_1) \cap J(z) \neq \emptyset.$ Since $\mathbb{X}$ is smooth, it follows that either  $J(y_1) =\{g\} $ or   $J(y_1) =\{-g\}.$ Then $|g(y_1)|= \|y_1\|,$ which contradicts the fact that $g\in \mathbb{Y}^{\perp}.$ Therefore, applying Theorem \ref{necessary; strong}, we conclude that $\mathbb{Y}$ is not a strongly anti-coproximinal subspace of $\mathbb{X}$. This completes the theorem.

	\end{proof}

We end this section with the following result, which is an immediate corollary of Theorem \ref{necessary; strong}.
	
	\begin{cor}	Let $\mathbb{Y}$ be a closed subspace of a Banach space $\mathbb{X}.$ Suppose that $\mathbb{X}$ satisfies either of the following properties:
		\begin{itemize}
			\item[(i)] $\mathbb{X}$ is a finite-dimensional smooth Banach space
			\item[(ii)] $\mathbb{X}$ is a finite-dimensional strictly convex Banach space
				\item[(iii)] $\mathbb{X}$ is a uniformly smooth Banach space.
		\end{itemize}
	 Then $\mathbb{Y}$ is not a strongly anti-coproximinal subspace of $\mathbb{X}$.
    \end{cor}
	

\section*{Section-II}
In this section we study the anti-coproximinal and the strongly anti-coproximinal subspaces in finite-dimensional polyhedral Banach spaces. In particular, we obtain characterizations of the said subspaces in $\ell_\infty^n$ and $\ell_1^n,$ which are computationally effective. 
We begin with the following characterization of the anti-coproximinal subspaces of a finite-dimensional polyhedral Banach space, provided that  $ Sm(\mathbb{X}) \cap\mathbb{{Y}} $ is dense in $\mathbb{{Y}}.$

\begin{theorem}\label{theorem}
		Let $ \mathbb{Y}$ be a subspace of an $n$-dimensional polyhedral Banach space $\mathbb{X}$ such that $ Sm(\mathbb{X}) \cap\mathbb{Y} $ is dense in $\mathbb{{Y}}.$ Then $ \mathbb{Y}$ is an anti-coproximinal subspace of $\mathbb{X}$ if and only if there are $n$-linearly independent elements in $\mathcal{J}_\mathbb{Y}.$
	\end{theorem}

\begin{proof}
	We first prove the sufficient part. Suppose on the contrary that $\mathbb{Y}$ is not an anti-coproximinal subspace of $\mathbb{X}$. Then there exists an element $x \in \mathbb{X}\setminus \mathbb{Y}$ and  a $y_0 \in \mathbb{Y}$ such that $y_0 \in \mathcal{R}_\mathbb{Y}(x).$ Therefore, from Theorem \ref{PS},  it follows  that $x-y_0 \in ker~f,$ for all $f \in \mathcal{J}_\mathbb{Y}.$ Since $\mathcal{J}_\mathbb{Y}$ contains $n$ linearly independent elements and $\dim (\mathbb{X}^*) = n$, it follows that $\cap_{f\in \mathcal{J}_\mathbb{Y}} ker~f = \{0\}.$ Therefore, $x-y_0 = 0,$ i.e., $x=y_0,$ which is a contradiction. This completes the proof of the sufficient part.
	
	To prove the necessary part, we first show that for any $y\in S_\mathbb{Y},$ $\mathcal{J}_\mathbb{Y} \cap J(y) \neq \emptyset.$ Since $Sm(\mathbb{X}) \cap \mathbb{Y}$ is dense in $\mathbb{Y},$ it is immediate that  $Sm(\mathbb{X}) \cap S_\mathbb{Y}$ is dense in $S_\mathbb{Y},$ and therefore, for any $y\in S_\mathbb{Y},$ there exists a sequence $\{y_n\}_{n\in \mathbb{N}} \subset Sm(\mathbb{X}) \cap S_\mathbb{Y}$ such that $y_n \to y.$ Suppose that $J(y_n) = \{f_n\} \subset \mathcal{J}_\mathbb{Y},$ for each $n \in \mathbb{N}.$ Therefore, it is easy to see that $f_n(y) \to 1.$ Since $\mathbb{X}$ is polyhedral and $\mathcal{J}_\mathbb{Y} \subset Ext(B_{\mathbb{X}^*})$, it follows that  $\mathcal{J}_\mathbb{Y}$ is finite. Moreover, since for each $n \in \mathbb{N},$ $f_n\in \mathcal{J}_\mathbb{Y}$ and $\mathcal{J}_\mathbb{Y}$ is finite, we get that for some  $f \in \mathcal{J}_\mathbb{Y},$ $f(y) =1.$ Thus $f \in \mathcal{J}_\mathbb{Y} \cap J(y). $   If possible, suppose that $\mathcal{J}_\mathbb{Y}$ contains exactly $k$ linearly independent elements such that $k < n.$ Let us assume that $(\neq 0) x \in \cap_{f\in \mathcal{J}_\mathbb{Y}} ker~f.$ Clearly, $x \notin \mathbb{Y}.$ Otherwise, from the above arguments $f(x)=1,$ for some $f \in \mathcal{J}_\mathbb{Y}.$ Since for each $y\in S_\mathbb{Y},$ $\mathcal{J}_\mathbb{Y} \cap J(y) \neq \emptyset,$ it follows  that there exists an $f \in J(y)$ such that $f(x)=0.$ It is immediate from Theorem \ref{PS} that $0 $ is a best coapproximation to $x$ out of $\mathbb{Y},$ which is not possible  since $\mathbb{Y}$ is an anti-coproximinal subspace of $\mathbb{X}$. This completes the proof of the necessary part. Hence the theorem.
\end{proof}

We note that the sufficient part of the above theorem holds true for any finite-dimensional Banach space $\mathbb{X}$ without any additional assumptions. In particular, this implies that \emph{ if $\mathbb{Y}$ is a subspace of an $n$-dimensional Banach space $\mathbb{X}$ such that there are $n$-linearly independent elements in $\mathcal{J}_{\mathbb{Y}},$ then $\mathbb{Y}$ is an anti-coproximinal subspace of $\mathbb{X}.$} 

We next characterize the strongly anti-coproximinal subspaces in finite-dimensional polyhedral Banach spaces.



\begin{theorem}\label{extreme}
	Let $\mathbb{Y}$ be a subspace of a finite-dimensional polyhedral Banach space $\mathbb{X}.$  Then the following statements are equivalent:
	\begin{itemize}
		\item[(i)] $\mathbb{Y}$ is a strongly anti-coproximinal subspace of $\mathbb{X}$
		\item[(ii)] $ \mathbb{{Y}}$ intersects the interior of every  facet of $ B_\mathbb{X}$
		\item[(iii)]  $\mathcal{J}_{\mathbb{Y}} = Ext(B_{\mathbb{X}^*}).$
	\end{itemize}
	
\end{theorem}

\begin{proof}
	Suppose that  $\pm F_1, \pm F_2, \ldots, \pm F_r$ are the facets  of $ B_{\mathbb{{X}}}.$  Following from  \cite[Lemma 2.1]{SPBB}, assume that $\pm f_1, \pm f_2, \ldots, \pm f_r$ are the corresponding extreme points of $B_{\mathbb{X}^*},$ respectively. Clearly, $Ext(B_{\mathbb{X}^*})= \{\pm f_1, \pm f_2, \ldots, \pm f_r\}.$ We complete the proof in the following three steps:\\
	(i) $\implies$ (ii): Suppose on the contrary that $\mathbb{Y}$ does not intersect the interior of the facet $F_i,$ for some $i \in \{1, 2, \ldots, r\}.$ Take $x \in int(F_i).$ Clearly, $x \notin \mathbb{Y}.$ Moreover, $x$ is smooth and $J(x)= \{f_i\}.$ Since $\mathbb{Y}$ does not intersect the interior of $F_i,$ it is easy to observe that for any $y \in \mathbb{Y},$ $J(y) \cap (Ext(B_{\mathbb{X}^*}) \setminus \{ \pm f_i\}) \neq \emptyset,$ otherwise, $J(y)=\{f_i\}.$ Take $\epsilon_0= \max\{ |f(x)|: f \in Ext(B_{\mathbb{X}^*}) \setminus \{ \pm f_i\}\}.$ As $J(x)=\{f_i\},$ it is clear that $\epsilon_0 <1.$ Since for any $y \in \mathbb{Y},$ $J(y) \cap (Ext(B_{\mathbb{X}^*}) \setminus \{ \pm f_i\}) \neq \emptyset,$ it is easy to see that for any $y \in \mathbb{Y}$ there exists an $f \in J(y)$ such that $|f(x)| \leq \epsilon_0.$ Following Theorem \ref{chm},  $y \perp_B^ {\epsilon_0} x,$ for any $y \in \mathbb{Y}.$  In other words, $0$ is an $\epsilon_0$-best coapproximation to $x$ out of $\mathbb{Y}.$ This contradicts that $\mathbb{Y}$ is a strongly anti-coproximinal subspace of $\mathbb{X}.$\\
	(ii) $\implies $ (iii): Suppose that $y_i \in int(F_i) \cap \mathbb{Y},$ for each $i \in \{1, 2, \ldots, r\}$. Clearly, $y_i$ is smooth and  $J(y_i)= \{f_i\}.$ Therefore, $f_i \in \mathcal{J}_{\mathbb{Y}}$,  for each $i \in \{1, 2, \ldots, r\},$ this implies that $Ext(B_{\mathbb{X}^*}) \subset \mathcal{J}_{\mathbb{Y}}.$ So, $Ext(B_{\mathbb{X}^*}) = \mathcal{J}_{\mathbb{Y}}.$ \\
(iii) $\implies$ (i): Let $x \in \mathbb{X}.$ Without loss of generality we assume that $x \in F_i,$ for some $i \in  \{1, 2, \ldots, r\}.$ Clearly, $f_i \in J(x).$ Since $\mathcal{J}_{\mathbb{Y}}= Ext(B_{\mathbb{X}^*}),$ there exists a $y \in Sm(\mathbb{X}) \cap \mathbb{Y}$ such that $J(y)=\{f_i\}.$ Therefore, $J(y) =\{f_i\} \subseteq J(x).$  By applying Theorem \ref{sufficient; strong},   we obtain that $\mathbb{Y}$ is a strongly anti-coproximinal subspace of $\mathbb{X}$.
	\end{proof}

 In the following example we explicitly show the applicability of the previous two theorems.

\begin{example}\label{example2}
	Let $ \mathbb{X}$ be the $3$-dimensional Banach space, where 
	$ B_{\mathbb{{X}}} $ is a hexagonal right prism with vertices $\pm (1, 0, 1),  \pm (- 1, 0,  1), \pm (\frac{1}{2}, \frac{1}{2}, 1), \pm (-\frac{1}{2}, \frac{1}{2}, 1), \pm (-\frac{1}{2}, -\frac{1}{2}, 1),$ $ \pm (\frac{1}{2}, -\frac{1}{2}, 1).$ Let $ f_1(x, y, z)= x+y, ~ f_2(x,y,z)= x-y,~ f_3(x,y,z)= y,~ f_4(x, y, z)= z, $ for any $ (x, y, z) \in \mathbb{{X}}.$ It is easy to observe that $Ext(B_{\mathbb{{X}}^*})=\{\pm f_1, \pm f_2, \pm f_3, \pm f_4\}.$ Let us assume that $\pm F_1, \pm F_2, \pm F_3$ and $\pm F_4 $ are the corresponding facets of $B_{\mathbb{X}}$ of the extreme functionals $\pm f_1, \pm f_2, \pm f_3, \pm f_4,$ respectively \cite[Lemma 2.1]{SPBB}. Clearly, $\pm F_1, \pm F_2, \pm F_3$ are the rectangular facets and $\pm F_4$ are the hexagonal facets of $B_{\mathbb{{X}}}.$ \\
		\textbf{Coproximinal subspace:}  Let $\mathbb{{Y}}= span\{(1, 0, 0), (0, 1, 0)\}.$ It is now easy to observe that for any $y \in \mathbb{Y},$ the elements of $J(y)$ can be written as a convex combination of $ \pm f_1, \pm f_2,  \pm f_3.$ Otherwise, there exists an element $z \in \mathbb{Y}$ such that $f_4 \in J(z),$ which is clearly not possible. Let $\widetilde{x}=(a,b,c) \in \mathbb{X}.$ It is straightforward to observe that $f_i((a,b,c)- (a,b,0))=0,$ for any $i \in \{1,2,3\}.$  Therefore for any $y \in \mathbb{Y}$, there exists $f_y \in J(y)$ such that $f_y( (a,b,c)-(a,b,0))=0,$ and consequently, $\mathbb{Y} \perp_B (a,b,c)-(a,b,0).$ This implies $(a,b,0)$ is a best coapproximation to $(a, b,c)$ out of $\mathbb{Y}.$ Therefore, $\mathbb{Y}$ is a coproximinal subspace of $\mathbb{X}.$ \\  
		\textbf{Anti-coproximinal subspace:} Let us now take $ \mathbb{{Y}}= span\{ ( \frac{3}{4}, -\frac{1}{4}, 1), ( -\frac{3}{4}, -\frac{1}{4}, 1)\}.$
	It is easy to observe that	$Sm(\mathbb{X}) \cap S_{\mathbb{{Y}}}$ is  dense in $ S_{\mathbb{{Y}}}$ and 
	$ \mathbb{{Y}}$ intersects four rectangular facets, $\pm F_1, \pm F_2$ and two hexagonal facets, $ \pm F_4$ of $B_{\mathbb{{X}}}.$ It is evident that $f_1, f_2, f_4 \in \mathcal{J}_\mathbb{Y}.$ As $ f_1, f_2, f_4$ are linearly independent, by applying Theorem \ref{theorem}  we get that $\mathbb{{Y}}$ is an anti-coproximinal subspace of $\mathbb{{X}}.$ \\
		\textbf{Strongly anti-coproximinal subspace:} Let $\mathbb{Y} = span\{(\frac{7}{8}, \frac{1}{8}, 1), (\frac{7}{8} ,-\frac{1}{8}, 1)\}.$ 	 By a straightforward computation we obtain that $(\frac{11}{16}, \frac{5}{16}, \frac{11}{14}) \in int(F_1), (\frac{11}{16}, -\frac{5}{16}, \frac{11}{14}) \in int(F_2), (0, \frac{1}{2}, 0) \in int(F_3)$ and $(\frac{7}{8}, 0, 1) \in int(F_4).$ Therefore, it is easy to see that the subspace $\mathbb{Y}$ intersects the interior of each facet of $B_\mathbb{X}.$ Applying Theorem \ref{extreme}, we conclude that $\mathbb{Y}$ is a strongly anti-coproximinal subspace of $\mathbb{X}.$ 
	
\end{example}
We now study the anti-coproximinal and the strongly anti-coproximinal subspaces of $\ell_\infty^n.$ The best coapproximation problem in $\ell_\infty^n$ was studied in \cite{SSGP} using the concept of the $*$-Property which plays a crucial role in the whole scheme of things. For the convenience of the readers, let us recall the definition of the $*$-Property.

\begin{definition}\cite{SSGP}
	Let $\mathcal{A}=\{\widetilde{a}_1, \widetilde{a}_2, \ldots, \widetilde{a}_m\}$ be a set of linearly independent elements in $\ell_\infty^n,$ where $1 \leq m \leq n$ and $\widetilde{a}_k = (a_1^k, a_2^k, \ldots, a_n^k),$ for each $1\leq k \leq m.$
	\begin{itemize}
		\item[(i)] For each $i\in \{1, 2, \ldots, n\},$ the $i$-th component of $\mathcal{A}$ is defined as $(a_i^1, a_i^2, \ldots, a_i^m).$
		\item[(ii)] The positively associative set $P_i^+(\mathcal{A})$ of the $i$-th component is defined as:
		\[P_i^+(\mathcal{A}) := \{j \in \{1, 2, \ldots, n\} : (a_i^1, a_i^2, \ldots, a_i^m) = (a_j^1, a_j^2, \ldots, a_j^m)\}.\]
		Similarly, the negatively associated set $P_i^-(\mathcal{A})$ is defined as:
		\[P_i^-(\mathcal{A}) := \{j \in \{1, 2, \ldots, n\} : (a_i^1, a_i^2, \ldots, a_i^m) = -(a_j^1, a_j^2, \ldots, a_j^m)\}.\] We write $P_i^+(\mathcal{A})=P_i^+$ and $P_i^-(\mathcal{A})= P_i^-.$
		\item[(iii)] The $i$-th component of $\mathcal{A}$ is said to satisfy the $*$-Property if there exist $\beta_1, \beta_2, \ldots, \beta_m \in \mathbb{R}$ such that the following holds true:
		\[\bigg|\sum_{k=1}^{m} \beta_k a_i^k\bigg| > \max \bigg\{ \bigg|\sum_{k=1}^{m} \beta_k a_j^k\bigg|: j \in \{1,2, \ldots, n\} \setminus P_i^+ \cup P_i^-\bigg\}.\] 
		
	\end{itemize}
\end{definition}
We next recall the characterization of smooth points of the unit sphere of $\ell_{\infty}^n$ and the extreme points of the unit ball of $(\ell_{\infty}^n)^*.$ 

\begin{prop}\label{description}
	$\widetilde{x}=(x_1, x_2, \ldots, x_n) \in S_{\ell_{\infty}^n}$ is  a smooth point if and only if there exists $i_0 \in \{1, 2, \ldots, n\}$ such that $|x_{i_0}|=1$ and $|x_j|< 1,$ for each $j \in \{ 1,2, \ldots, n\} \setminus \{i_0\}.$ Moreover, $f \in ( \ell_{\infty}^n)^*$ is an extreme point of $B_{(\ell_{\infty}^n)^*}$ if and only if there exists $i_0 \in \{1, 2, \ldots, n\}$ such that  $f(x_1, x_2, \ldots, x_n)= x_{i_0}, $ for any $(x_1, x_2, \ldots, x_n) \in \ell_{\infty}^n.$
\end{prop}

The following lemma is essential to characterize the anti-coproximinal and the strongly anti-coproximinal subspaces in $\ell_\infty^n.$

\begin{lemma}\label{existence}
	Let $ \mathcal{A} = \{ \widetilde{a}_1 , \widetilde{a}_2 ,\ldots, \widetilde{a}_m\} $ be a set of linearly independent elements in $\ell_\infty^n$, where $  1 < m < n  $ and $\widetilde{a}_k=(a_1^k,  a_2^k, \ldots, a_n^k ), $ for each  $ 1 \leq k \leq m.$ Let $\mathbb{Y} = span~\mathcal{A}.$ If every component of $\mathcal{A}$ satisfies the $*$-Property and $ | P_i^+ \cup P_i^-| =1,$ for all $ 1 \leq i \leq n$ then for every $f \in Ext(B_{({\ell_\infty^n})^*}),$ there exists a $\widetilde{y} \in Sm(\mathbb{X}) \cap S_\mathbb{Y} $ such that $ J(\widetilde{y}) = \{f\}.$	
\end{lemma}

\begin{proof}
	Let $f_i\in (\ell_\infty^n)^* $ be such that $f_i (x_1, x_2, \ldots, x_n)= x_i.$ Clearly,  $Ext(B_{({\ell_\infty^n})^*})=\{ \pm f_1 , \pm f_2, \ldots, \pm f_n\}. $ 
	Since the $i$-th component satisfies the $*$-Property and $| P_i^+ \cup P_i^-| =1$ then there exist $\beta_1, \beta_2, \ldots, \beta_m \in \mathbb{R}$ such that 
	\[\big|\sum_{k=1}^{m} \beta_k a_i^k \big| > \max \bigg\{ \big|\sum_{k=1}^{m} \beta_k a_j^k\big|: j \in \{1,2, \ldots, n\} \setminus \{i\}\bigg\}.\]
	Clearly, $\|\sum_{k=1}^{m} \beta_k \widetilde{a_k}\|= |\sum_{k=1}^{m} \beta_k a_i^k|.$ Take $\widetilde{y}= \frac{\sum_{k=1}^{m} \beta_k \widetilde{a_k}}{\|\sum_{k=1}^{m} \beta_k \widetilde{a_k}\|}.$ Clearly, $|f_i(\widetilde{y})|=1.$ It is easy to observe from Proposition \ref{description} that $\widetilde{y}\in Sm(\mathbb{X}) \cap S_\mathbb{Y}.$  Thus  either $J(\widetilde{y}) = \{f_i\}$ or $J(-\widetilde{y}) = \{f_i\}.$ 
\end{proof}

We are now ready to present the desired characterization.
\begin{theorem}\label{no}
	Let $ \mathcal{A} = \{ \widetilde{a}_1 , \widetilde{a}_2 ,\ldots, \widetilde{a}_m\} $ be a set of linearly independent elements in $\ell_\infty^n$, where $  1 < m < n  $ and $\widetilde{a}_k=(a_{1}^k, a_2^k, \ldots, a_n^k ), $ for each  $ 1 \leq k \leq m. $ Let $ \mathbb{Y}= span~\mathcal{A}.$ Then the following statements are equivalent: 
	\begin{itemize}
		\item[(i)] $\mathbb{Y}$ is a strongly anti-coproximinal subspace of $\ell_{\infty}^n.$
		\item[(ii)] $\mathbb{Y}$ is an anti-coproximinal subspace of $\ell_{\infty}^n.$
		\item[(iii)] Every component of $\mathcal{A}$ satisfies the $*$-Property and $ | P_i^+ \cup P_i^-| =1,$ for all $ 1 \leq i \leq n.$ 
	\end{itemize}
	
\end{theorem}

\begin{proof}
	
We begin the proof by noting that (i) $\implies$ (ii)  follows trivially.\\
	Next we prove that (ii) $\implies$ (iii).
	Let $Ext(B_{({\ell_\infty^n})^*}) = \{\pm f_1, \pm f_2, \ldots, \pm f_n\}, $ where $f_i (x_1, x_2, \ldots, x_n)= x_i,$ for any $(x_1, x_2, \ldots, x_n)\in \ell_\infty^n.$
		Suppose on the contrary that for some $j \in \{1, 2, \ldots, n\},$ the $ j$-th component does not satisfy the $*$-Property. It is easy to observe that there does not exist any $y\in \mathbb{Y}$ such that $J(y)=\{f_j\}.$ Otherwise, considering $y= \sum_{k=1}^{m} \beta_k \widetilde{a}_k, $ for some $\beta_1, \beta_2, \ldots, \beta_m \in \mathbb{R},$ we obtain that $|\sum_{k=1}^{m} \beta_k a_j^k|=|f_j(y)|> |f_t(y)|= |\sum_{k=1}^{m} \beta_k a_t^k|, $ for any $t \in \{1, 2, \ldots, n\} \setminus \{j\},$ which contradicts that the $j$-th component does not satisfy the $*$-Property. Let us consider $\widetilde{b}=(b_1, b_2, \ldots, b_n)\in \ell_\infty^n,$ where $b_j=1$ and $b_k=0,$ for all $k \neq j.$ Clearly, $\widetilde{b} \notin \mathbb{Y}.$ Since $J(y)$ is a face of $B_{(\ell_{\infty}^n)^*}$ and  there does not exist any $y \in \mathbb{Y}$ such that $J(y)=\{f_j\},$ we conclude that for any $y \in \mathbb{Y},$ there exists an $f \in Ext(B_{({\ell_{\infty}^n})^*}) \setminus \{ \pm f_j\}$ such that $f \in J(y).$ As $f(\widetilde{b})=0,$ for any $f \in Ext(B_{{\ell_{\infty}^n}^*}) \setminus \{ \pm f_j\},$ by using Theorem \ref{PS} we obtain that $0$ is a best coapproximation to $\widetilde{b}$ out of $\mathbb{Y}.$ This contradicts that $\mathbb{Y}$ is an anti-coproximinal subspace of $\ell_{\infty}^n$ . 
		
 To obtain (iii), we now need to show that  $ |P_i^+ \cup P_i^-|=1,$ for all $1 \leq i \leq n.$ Clearly, $i \in P_i^+ \cup P_i^-.$ 
	Suppose on the contrary that  $l \in P_{i}^+ \cup P_{i}^-,$ for some $l \in \{1, 2, \ldots, n\} \setminus \{i\}.$ Consider the element  $\widetilde{b}= (b_1, b_2, \ldots, b_n)  \in \ell_\infty^n$, where $ b_l = 1$ and $ b_k = 0,$ for all $ k \neq l.$ Clearly, $\widetilde{b} \notin \mathbb{Y}$ and $f_k(\widetilde{b})=0,$ for any $k \in \{1, 2, \ldots, n\} \setminus \{l\}.$ As, $l \in P_{i}^+ \cup P_{i}^-,$ for any $y \in \mathbb{Y},$ $|f_l(y)|= |f_i(y)|.$ Therefore, if $f_l \in J(y),$ for some $y \in \mathbb{Y},$ then either $f_i \in J(y)$ or $-f_i \in J(y).$
	We conclude that for any $y \in \mathbb{Y},$ there exists $f \in Ext(B_{{\ell_\infty^n}^*}) \setminus \{\pm f_l\}$ such that $f \in J(y).$ As, $f(\widetilde{b})=0,$ for any $f \in Ext(B_{{\ell_\infty^n}^*}) \setminus \{\pm f_l\},$ by using Theorem \ref{PS} we obtain that $0$ is a best coapproximation to $\widetilde{b}$ out of $\mathbb{Y}.$ This again contradicts our hypothesis that $\mathbb{Y}$ is an anti-coproximinal subspace of $\ell_{\infty}^n$. This completes the proof.

			Let us now prove (iii) $\implies$ (i). Suppose on the contrary that $y_0$ is an $ \epsilon$-best coapproximation to $\widetilde{b} \in \ell_\infty^n \setminus \mathbb{{Y}}$ out of $ \mathbb{{Y}}.$ Following Proposition \ref{approximate}(iii), we note that for each $y \in \mathbb{Y}$ there exists an $f \in J(y)$ such that $|f(\widetilde{b} - y_0)| \leq \epsilon \|\widetilde{b} - y_0\|,$ for some $\epsilon \in [0, 1).$ 
	Since each component of $\mathcal{A}$ satisfying the $*$-Property and $|P_i^+ \cup P_i^-| =1,$ it follows from Lemma \ref{existence} that for every $f \in Ext(B_{({\ell_\infty^n})^*})$ there exists a $\widetilde{y} \in Sm(\mathbb{X}) \cap  S_\mathbb{Y} $ such that $ J(\widetilde{y}) = \{f\}.$  
	Therefore, $|f(\widetilde{b} -y_0)| \leq \epsilon \|\widetilde{b} - y_0\|,$ for any $f \in Ext(B_{({\ell_\infty^n})^*}).$ This implies that $|f(\widetilde{b} -y_0)| <  \|\widetilde{b} - y_0\|,$ for any $f \in Ext(B_{({\ell_\infty^n})^*}).$ Therefore, it is easy to observe that
	\[\|\widetilde{b} - y_0\| = \sup_{\|f\| \leq 1}\big\{|f(\widetilde{b} - y_0)|\big\} = \max_{ f \in Ext(B_{({\ell_\infty^n})^*})} \big\{|f(\widetilde{b} - y_0)|\big\}  <   \|\widetilde{b} - y_0\|,\] a contradiction. Hence (iii) $\implies$ (i).

\end{proof}

The following example illustrates the computational effectiveness of Theorem \ref{no} in verifying the strong anti-coproximinality of a given subspace of $ \ell_\infty^n. $

\begin{example}\label{example3}
	Let $ \widetilde{a}_1= (-4, 2, 3, 1, 3), \widetilde{a}_2= (1, -5, 4, 2, -3), \widetilde{a}_3=( 1, 3, -7, 4, 6) \in \ell_\infty^5$ and let $ \mathbb{{Y}}= span\{ \widetilde{a}_1, \widetilde{a}_2, \widetilde{a}_3\}.$ The $1$st, $2$nd, $3$rd, $4$th and the $ 5$th components are $ (-4, 1, 1), (2, -5, 3), (3, 4, -7), (1, 2, 4)$ and $(3, -3, 6),$ respectively. It is immediate that $ |P_i^+ \cup P_i^-|= 1,$ for all $ 1 \leq i \leq 5.$ It is straightforward to verify that every component satisfies the $*$-Property. Therefore, following Theorem  \ref{no}, $\mathbb{{Y}}$ is a strongly  anti-coproximinal subspace of $ \ell_\infty^5.$ 
\end{example}

In light of the above example, we make the following remark that emphasizes the importance of the mother space in deciding whether a given subspace is strongly anti-coproximinal or not.

\begin{remark}\label{rem}
	Let $\mathbb{X} = \ell_\infty^6$ and let $\{e_1, e_2, \ldots, e_6\}$ be the standard ordered basis of $\mathbb{X}.$ Suppose that $\mathbb{Z}= span \{e_1, e_2, \ldots, e_5\}$ is a subspace of $\mathbb{X}.$ Let us now define $\psi: \ell_{\infty}^5 \to \mathbb{Z},$ as $\psi(x_1, x_2, \ldots, x_5)= (x_1, x_2, \ldots, x_5, 0),$ for any $(x_1, x_2, \ldots, x_5) \in \ell_{\infty}^5.$ Clearly, $\psi$ is an isometric isomorphism.
		Let $\mathbb{Y}$ be the same subspace  given in Example \ref{example3}. As $\mathbb{Y}$ is strongly anti-coproximinal in $\ell_{\infty}^5,$  clearly we observe that $\psi(\mathbb{Y})$ is also strongly anti-coproximinal  in $\mathbb{Z}.$ However, $\psi(\mathbb{Y})$ is not an anti-coproximinal subspace  of $\mathbb{X},$ as $\psi(\widetilde{a}_1)= (-4, 2, 3, 1, 3, 0), \psi(\widetilde{a}_2)= (1, -5, 4, 2, -3, 0), \psi(\widetilde{a}_3)= (1, 3, -7, 4, 6, 0)$ and the $6$th component does not satisfy the $*$-Property (see Theorem \ref{no}).  
\end{remark}

As promised earlier, here we present an example of a subspace which is not strongly anti-coproximinal but satisfies the necessary condition of Theorem \ref{necessary; strong}.

\begin{example}\label{not strong}
	Let $\widetilde{a}_1= (1, 1, 2), \widetilde{a}_2=(2, 2, 1) \in \ell_\infty^3$ and let $\mathbb{Y}= span \{\widetilde{a}_1, \widetilde{a}_2\}.$ The $1$st, $2$nd, $3$rd components are $ (1, 2), (1, 2), (2, 1),$ respectively. It is immediate that $ |P_1^+ \cup P_1^-|= 2.$ Therefore, following Theorem  \ref{no}, $\mathbb{{Y}}$ is not a strongly  anti-coproximinal subspace of $ \ell_\infty^3.$ However, by a straightforward observation we can verify that for any $x \in \ell_\infty^3,$ there exists a $y \in \mathbb{Y}$ such that $J(y) \cap J(x) \neq \emptyset.$
\end{example}

  In \cite{SSGP2}, the authors introduced the notions `the zero set($\mathcal{Z}_{\mathcal{A}}$)'    and `the minimal norming set($\mathcal{N}$)' to study the best coapproximation problem in $\ell_1^n.$ Using these notions, we  study the anti-coproximinal subspaces and the strongly anti-coproximinal subspaces in $\ell_1^n.$ Let us now recall the definitions of the zero set and the minimal norming set.

\begin{definition}\cite{SSGP2}
	Let $\mathcal{A}=\{\widetilde{a}_1, \widetilde{a}_2, \ldots, \widetilde{a}_m\}$ be a set of linearly independent elements in $\ell_1^n,$ where $1 \leq m \leq n$ and $\widetilde{a}_k = (a_1^k, a_2^k, \ldots, a_n^k),$ for each $1\leq k \leq m.$ The \emph{zero set} $ \mathcal{Z}_{\mathcal{A}}$ of $\mathcal{A}$ is defined as 
	\[
	\mathcal{Z}_{\mathcal{A}} = \Big\{  i \in \{1, 2, \ldots, n\} : \left( a_{i} ^1 ,a_{i}^2 ,\ldots,a_{i} ^m \right) =  (0, 0, \ldots, 0) \Big\}.
	\]
\end{definition}

\begin{definition}\cite{SSGP2}\label{minimal}
	A set $S$ in a Banach space $\mathbb{X}$ is said to be   symmetric  if $x \in S$ implies  $-x \in S.$
	Let $\mathbb{{Y}}  $ be a subspace of $ \ell_1^n.$  A symmetric set $\mathcal{N}$ is said to be 
	a norming set  of $ \mathbb{{Y}}$ if  $\left (M_g \cap Ext(B_{\ell_\infty^n}) \right)  \cap  \mathcal{N} \neq \emptyset,$ for each  $g\in \psi(\mathbb{{Y}}),$ where $\psi$ is the canonical isometric isomorphism between $\ell_1^n$ and $(\ell_\infty^n)^{*}.$ A norming set  $\mathcal{N}$  is said to be  a \textit{minimal norming set} of  $ \mathbb{{Y}}$ if for any norming set $\mathcal{M}$ of $ \mathbb{{Y}},$ $\mathcal{M} \subset \mathcal{N}$ implies that  $\mathcal{M}= \mathcal{N}.$
\end{definition}

In the following theorem we completely characterize the anti-coproximinal subspaces in $\ell_1^n$ in terms of the minimal norming set. As we will observe, this characterization turns out to be particularly helpful for explicitly describing the anti-coproximinal subspaces in $\ell_{1}^n.$

\begin{theorem}\label{anti}
	Let $\mathbb{Y}$ be a subspace of $\ell_1^n$ and let $\mathcal{A}=\{\widetilde{a}_1, \widetilde{a}_2, \ldots, \widetilde{a}_m\} $ be a basis of $\mathbb{Y},$ where $1< m <n$ and $\widetilde{a}_k =(a_1^k, a_2^k, \ldots, a_n^k),$ for each $1 \leq k \leq m.$ Then the following statements hold true:
	\begin{itemize}
		\item[(i)]  If $\mathcal{Z}_\mathcal{A}\neq \emptyset, $ then $\mathbb{Y}$ is not an anti-coproximinal subspace of $\ell_{1}^n.$
		\item[(ii)] If $\mathcal{Z}_\mathcal{A} = \emptyset,$ then $\mathbb{Y}$ is an anti-coproximinal subspace of $\ell_{1}^n$ if and only if $dim(span~\mathcal{N})=n,$ where $\mathcal{N}$ is the minimal norming set.
	\end{itemize}
\end{theorem}

\begin{proof}
	(i) Suppose that $j \in \mathcal{Z}_\mathcal{A},$ i.e., $a_j^k = 0,$ for all $k \in \{1, 2, \ldots, n\}.$ Let $\widetilde{b}= (b_1, b_2, \ldots, b_n) \in \ell_1^n,$ where $b_j=1$ and $b_i=0,$ for any $i \in \{1, 2, \ldots, n\} \setminus \{j\}.$  Clearly, $\widetilde{b} \notin \mathbb{Y}.$ For any $\beta_1, \beta_2, \ldots, \beta_n \in \mathbb{R},$ it is easy to see that 
	\[\|\widetilde{b} - \sum_{i=1}^m \beta_i \widetilde{a}_i\| = \|\sum_{i=1}^m \beta_i \widetilde{a}_i\| + 1 > \|\sum_{i=1}^m \beta_i \widetilde{a}_i - 0\|.\] Therefore, $0$ is a best coapproximation to $\widetilde{b}$ out of $\mathbb{Y}$ and consequently $\mathbb{Y}$ is not an  anti-coproximinal subspace of $\ell_{1}^n$.\\
	
	(ii) Since $\mathcal{Z}_\mathcal{A} = \emptyset,$ it follows from \cite[Th 2.2]{SSGP2} that the minimal norming set $\mathcal{N}$ is unique. Suppose that $\mathcal{N} = \{ \pm \widetilde{x}_1, \pm \widetilde{x}_2, \ldots, \pm \widetilde{x}_q\},$ where $\widetilde{x}_k=(x_1^k, x_2^k, \ldots, x_n^k).$ Without loss of generality we assume that $ \{ \widetilde{x}_1, \widetilde{x}_2, \ldots, \widetilde{x}_q\}$ is linearly independent.\\
	Let us first prove the necessary part. Suppose on the contrary that $q < n.$ It is straightforward to see that there exists a non zero element $\widetilde{b} = (b_1, b_2, \ldots, b_n) \in \ell_1^n$ such that $\sum_{i=1}^{n} b_i x_i^p=0,$ for any $1 \leq p \leq q.$ Observe that $\widetilde{b} \notin \mathbb{Y}.$ Otherwise,   there exists an $\widetilde{x}_k \in \mathcal{N}$  such that $\widetilde{x}_k \in M_{\psi(\widetilde{b})},$ for some  $1 \leq k \leq q,$ i.e., \[\psi(\widetilde{b})(\widetilde{x}_k) = \sum_{i=1}^{n} b_i x_i^k \neq 0, \] 
	where $\psi: \ell_1^n \rightarrow (\ell_\infty^n)^{*}$ is the canonical isometric isomorphism.
	Therefore, for $\alpha_1= \alpha_2= \ldots= \alpha_m =0,$ the following system of linear equations holds true:
	\begin{eqnarray*}
		\alpha_1\sum_{i=1}^n a_i^1 x_i^p + \alpha_2\sum_{i=1}^n a_i^2 x_i^p + \ldots + \alpha_m\sum_{i=1}^n a_i^m x_i^p = 0= \sum_{i=1}^{n} b_i x_i^p,
	\end{eqnarray*}
	where $p \in \{ 1, 2, \ldots, q\}.$ Now applying Theorem \cite[Th. 2.4]{SSGP2}, it is easy to observe that $0$ is a best coapproximation to $\widetilde{b}$ out of $\mathbb{Y}.$ This contradicts the fact that $\mathbb{Y}$ is an anti-coproximinal subspace of $\ell_{1}^n$. Thus we obtain $q=n.$ This proves the necessary part.\\
	
	We now prove the sufficient part.  Let $\mathcal{N}$ be the minimal norming set and let $|\mathcal{N}|=q.$  Suppose on the contrary that $\mathbb{Y}$ is not an anti-coproximinal subspace of $\ell_{1}^n$. Then there exists $\widetilde{b} = (b_1, b_2, \ldots, b_n) \in \ell_1^n \setminus \mathbb{Y}$ such that $y_0$ is a best coapproximation to $\widetilde{b}$ out of $\mathbb{Y}.$ It follows that $0$ is a best coapproximation to $\widetilde{b} - y_0$ out of $\mathbb{Y}.$ Assume that $\widetilde{b} - y_0 = \widetilde{d} = (d_1, d_2, \ldots, d_n).$ Clearly, $\widetilde{d} \in \ell_1^n \setminus \mathbb{Y}.$ Applying Theorem \cite[Th. 2.4]{SSGP2}, we obtain that $\sum_{i=1}^n d_ix_i^p = 0,$  for each $1 \leq p \leq q.$ 
	Since $dim(span~\mathcal{N})= dim(span \{ (x_1^p, x_2^p, \ldots, x_n^p): 1 \leq p \leq q\})=n,$ it follows that $\widetilde{d}=(d_1, d_2, \ldots, d_n)=0.$ 
	This contradicts that $\widetilde{d} \in \ell_1^n \setminus \mathbb{Y}.$ This completes the proof.

\end{proof}

Next we give an example of an anti-coproximinal subspace of $\ell_1^3$ by applying Theorem \ref{anti}.

\begin{example}
	Let $\mathcal{A} = \{(0, 1, 1), (-1, 0, 1)\} \subset \ell_1^3$ and  let $\mathbb{Y} = span~\mathcal{A}.$ It is easy to observe that $\mathcal{Z}_\mathcal{A} = \emptyset.$ Then following the same technique as given in the proof of \cite[Th. 2.2]{SSGP2}, we obtain that the minimal norming set $\mathcal{N}$ of $\mathbb{Y}$ as $\{\pm (1,1,1), \pm (-1, 1, 1), \pm(-1, -1, 1)\}.$ Thus we get $dim(span~\mathcal{N}) =3.$ Applying Theorem \ref{anti}(ii) we conclude that $\mathbb{Y}$ is an anti-coproximinal subspace of $\ell_{1}^n$.
    
\end{example}

Using Theorem \ref{extreme} we show that there does not exist any strongly anti-coproximinal subspace in $\ell_1^n.$ Before proving the result, we note the following lemma.

\begin{lemma}\label{equal}
	Let $\mathbb{Y}$ be a strongly anti-coproximinal subspace of $\ell_1^n.$ Then $\phi(\mathcal{J}_\mathbb{Y})=\mathcal{N},$ where $\mathcal{N}$ is the minimal norming set and  $\phi$ is the canonical isometric isomorphism between $(\ell_1^n)^*$ and $\ell_\infty^n.$ 
\end{lemma}
 
\begin{proof}
	
	Since $\mathbb{Y}$ is a strongly anti-coproximinal subspace of $\ell_1^n$, it follows from Theorem \ref{anti} that $\mathcal{Z}_\mathcal{A} = \emptyset.$  Suppose that $\psi : \ell_1^n \to (\ell_\infty^n)^{*}$ is the canonical isometric isomorphism. Let us first assume that $f\in \mathcal{J}_\mathbb{Y}.$ This implies that there exists a $y\in Sm(\ell_1^n) \cap S_\mathbb{Y}$ such that $f(y)=1.$ Now observe that $\psi(y)(\phi(f)) =  f(y)=1.$ As $\psi(y)$ is a smooth point in $(\ell_{\infty}^n)^*$, we have $ M_{\psi(y)}= \{\pm \phi(f)\}.$ Therefore, $\phi(\mathcal{J}_\mathbb{Y}) \subset \mathcal{N}.$  On the other hand, since $\mathbb{Y}$ is a strongly anti-coproximinal subspace of $\ell_{1}^n$, it follows from Theorem \ref{extreme} that $\mathcal{N} \subset Ext(B_{\ell_\infty^n}) = \phi(Ext(B_{(\ell_1^n)^*})) = \phi(\mathcal{J}_\mathbb{Y}).$ This proves our lemma.	
\end{proof}

Let us now present the desired result.

\begin{theorem}\label{no strong}
	There is no strongly anti-coproximinal subspace in $ \ell_1^n.$
\end{theorem}

\begin{proof}
	Suppose on the contrary we  assume that $\mathbb{{Y}} $ is a  strongly anti-coproximinal subspace in $ \ell_1^n.$  Clearly, $\mathbb{Y}$ is a proper subspace of $\ell_{1}^n.$ Let   $\mathcal{A}=\{\widetilde{a_1}, \widetilde{a_2}, \ldots, \widetilde{a_m}\}$ be a basis of $ \mathbb{{Y}},$ where $\widetilde{a_k} = (a_1^k, a_2^k, \ldots, a_n^k),$ for all $k\in \{1, 2, \ldots, m\},$ where $1 < m < n$. It follows from Theorem \ref{anti} that $\mathcal{Z}_\mathcal{A}=\emptyset.$ Also  from Theorem \ref{extreme}, we have $|\mathcal{J}_{\mathbb{Y}}|= |Ext(B_{(\ell_1^n)^*})|=2^n.$ By virtue of Lemma \ref{equal}, we note that the cardinality of the minimal norming set of $\mathbb{{Y}}$ is equal to the cardinality of the set $\mathcal{J}_\mathbb{{Y}}.$
	Let us now consider the following hyperspaces in $\mathbb{R}^m$ 
	\[
	H_i = \bigg\{(\beta_1, \beta_2, \ldots, \beta_m)\in \mathbb{R}^m : \sum_{k=1}^m\beta_ka_{i}^k = 0\bigg\},
	\] where $i \in \{1, 2, \ldots, n\}.$
	Now it is easy to observe from  \cite[Th. 2.2]{SSGP2} that there is an one-one correspondence between the minimal norming set of $\mathbb{{Y}}$ and the set of regions $\mathbb{R}^m$ formed  by $H_i$'s for all $i\in\{1, 2, \ldots, n\}.$ 
	Applying \cite[Th. 1]{HZ}, we obtain that these hyperspaces divide $ \mathbb{R}^m$ into at most $ 2 [ \binom{n-1}{0} + \binom{n-1}{1}+ \ldots + \binom{n-1}{m-1}]$ regions.  Therefore the cardinality of the minimal norming set of $\mathbb{Y}$ is at most  $ 2 [ \binom{n-1}{0} + \binom{n-1}{1}+ \ldots + \binom{n-1}{m-1}]$.  Since $ m < n,$ it is immediate that
	\[ 2 \bigg[ \binom{n-1}{0} + \binom{n-1}{1}+ \ldots + \binom{n-1}{m-1}\bigg] < 2^n.
	\]
	This contradicts that  the cardinality of the minimal norming set of $\mathbb{Y}$ is $ 2^n.$ This completes the proof of the theorem.
\end{proof}

We would like to  end this article with the following  remarks regarding the anti-coproximinal subspaces and the strongly anti-coproximinal subspaces.
\begin{remark}
(i)   We have already shown that there exists a strongly anti-coproximinal subspace in $ \ell_\infty^n ( n \geq 3),$ whereas there does not exist any strongly anti-coproximinal subspace in $\ell_1^n (n \geq 3).$ Using Theorem \ref{extreme}, it is possible to give a geometric interpretation of this phenomenon in a visually appealing manner. Indeed, we observe that there is a subspace in $\ell_\infty^n$ which intersects  the interior of each of the facet of $B_{\ell_\infty^n},$ whereas $\ell_1^n$ does not contain any such subspaces (see Theorem \ref{no strong}).\\
(ii) We note that the anti-coproximinal subspaces are the least favorable  subspaces from the perspective of best coapproximation. In this study we have observed that in general there may be many subspaces in a polyhedral Banach spaces which are anti-coproximinal. This further illustrates the non-triviality and the computational difficulty associated with the best coapproximation problem, even in finite-dimensional Banach spaces.  
We note from \cite[Th. 4]{J2} that a Banach space $\mathbb{X}$ having three or more dimension is an inner product space if and only if for each hyperspace $H$ of $\mathbb{X},$ there exists an $x \in \mathbb{X}$ such that $H\perp_B x.$ In any Banach space, it is easy to verify that the anti-coproximinal hyperspaces are precisely those which are not orthogonal to any element of $\mathbb{X}.$   Moreover, a strongly anti-coproximinal hyperspace $H$ is precisely those for which there does not exist any element $x \in \mathbb{X} $ such that that $H \perp_B^\epsilon x,$ for some $\epsilon \in [0,1).$
From Theorem \ref{no strong}, for any hyperspace $H$ of $\ell_1^n,$ we note that there exists an $\epsilon \in [0,1)$ and an $x \in \ell_1^n$ such that $H \perp_B^{\epsilon} x.$ \\
(iii) It is known \cite{LT, PS} that given a subspace
$\mathbb{Y}$ of a Banach space $\mathbb{X}$ and an element $x \notin \mathbb{Y}$, $y_0$ is a best coapproximation to $x$ out of
$\mathbb{Y}$ if and only if there exists a norm one projection from $span\{x, \mathbb{Y}\}$  to $\mathbb{Y}.$ It is clear that
\emph{$\mathbb{Y}$ is an anti-coproximinal subspace in $\mathbb{X}$ if and only if  given any $x \notin \mathbb{Y}, $ there does not exist any norm one projection from $span\{x, \mathbb{Y}\}$  to $\mathbb{Y}.$  In other words, $\mathbb{Y}$ is an anti-coproximinal subspace in $\mathbb{X}$ if and only if for any subspace $\mathbb{Z}$ which properly contains $\mathbb{Y},$ there does not exist any norm one projection from $\mathbb{Z}$ to $\mathbb{Y}.$\\}

\end{remark}

\noindent {\textbf{Data availability statement.}}\\
Data sharing is not applicable to this article as no new data were created or analyzed in this study.\\

\noindent {\textbf{Disclosure statement.}}\\
There are no relevant financial or non-financial competing interests to report.


\begin{thebibliography}{99}


\bibitem{B} Birkhoff, G.,  \textit{Orthogonality in linear metric spaces}, Duke Math. J. \textbf{1} (1935), 169--172.

\bibitem{Bruck} Bruck Jr, R. E., \textit{Property of fixed point sets of nonexpansive mappings in Banach spaces}, Trans. Amer. Math. Soc., \textbf{179} (1973), 251 - 262. 

\bibitem{Bruck1} Bruck Jr, R. E.,  \textit{Nonexpansive projections onto subsets of Banach spaces}, Pacific. J. Math., \textbf{47} (2) (1973), 341 - 355.








\bibitem{C05}  Chmieli\'nski, J., 
\textit{On an $\epsilon-$Birkhoff orthogonality},
Journal of Inequalities in Pure and Applied Mathematics,
\textbf{6}(3) (2005) Article 79.	

\bibitem{C} Chmieli\'nski, J., \textit{Approximate Birkhoff-James
	Orthogonality in Normed Linear Spaces
	and Related Topics}, Operator and norm inequalities and related topics, 303–320,
Trends in Math., Birkhäuser/Springer, Cham, 2022.


\bibitem{D} Dragomir, S. S., \textit{On approximation of continuous linear functionals in normed linear spaces}, An. Univ. Timişoara Ser. Ştiinţ. Mat., \textbf{29} (1991), 51–58. 


\bibitem{FF} Franchetti, C., Furi, M., \textit{Some characteristic properties of real Hilbert spaces},  Rev. Roumaine Math. Pures Appl., \textbf{17} (1972), 1045-1048.

\bibitem{G} Giles, J. R., \textit{Strong differentiability of the norm and rotundity of the dual}, J. Austral. Math. Soc., Ser. A, \textbf{26} (1978), 302–308.


\bibitem{HZ} Ho, C., Zimmerman, S., \textit{On the number of regions in an $m$-dimensional space cut by $n$ hyperplanes}, Austral. Math. Soc. Gaz., \textbf{33} (2006), 260-264.



\bibitem{J}James, R. C., \textit{Orthogonality and linear functionals in normed linear spaces}, Trans.  Amer. Math. Soc., \textbf{61} (1947), 265-292.

\bibitem{J2}	James, R. C., \textit{Inner products in normed linear spaces}, Bull. Amer. Math Soc., \textbf{53} (1947),
559-566.


\bibitem{KL} Kamiska, A., Lewicki, G., \textit{Contractive and optimal sets in Banach
spaces}, Math. Nachr., \textbf{268} (2004), 74-95.

\bibitem{LT} Lewicki, G., Trombetta, G., \textit{Optimal and one-complemented subspaces}, Monatsh. Math., \textbf{153} (2008), 115-132.

\bibitem{M} Megginson, R. E., \textit{An introduction to Banach space theory}, Graduate Texts in Mathematics, 183 (1998) Springer-Verlag, New York.

\bibitem{N} Narang, T. D., \textit{On best coapproximation in normed linear spaces}, Rocky Mountain J. Math., \textbf{22} (1992), 265-287.

\bibitem{PS} Papini, P. L., Singer, I., \textit{Best coapproximation in normed linear spaces}, Monatsh Math., \textbf{88} (1979), 27-44.




\bibitem{RS} Rao, G.S., Swaminathan, M., \textit{Best coapproximation and schauder bases in Banach spaces}, Acta Sci. Math., \textbf{54} (1990), 339-359.


\bibitem{R} Rudin, W., \textit{Functional analysis}, McGraw-Hill, Inc., 1991.



\bibitem{SPBB} Sain, D.,  Paul, K., Bhunia, P., Bag, Santanu., \textit{On the numerical index of polyhedral Banach space}, Linear Algebra Appl., \textbf{577} (2019), 121-133.



\bibitem{SSGP} Sain, D., Sohel, S., Ghosh, S., Paul, K.,  \textit{On best coapproximations in  subspaces of diagonal matrices}, Linear Multilinear Algebra, \textbf{71} (2023), 47-62.


\bibitem{SSGP2} Sain, D., Sohel, S., Ghosh, S., Paul, K., \textit{On Best coapproximation problem in $\ell_1^n$},  Linear Multilinear Algebra, DOI: 10.1080/03081087.2022.2153101.


\bibitem{W} Westphal, U.,  \textit{Cosuns in $\ell_p(n)$}, J. Approx. Theory, \textbf{54} (1988), 287 - 305.



\end{thebibliography}
\end{document}